\documentclass{amsart}
\usepackage{amsmath}
\usepackage{amsthm}
\usepackage{amssymb}
\usepackage[curve]{xypic}
\usepackage{dsfont}
\usepackage{latexsym}
\usepackage{varioref}
\originalTeX
\usepackage{cite}
\usepackage{overpic}

\def\IN{\mathds{N}}
\def\IR{\mathds{R}}

\def\Inv{\mathrm{Inv}}

\def\eps{\varepsilon}
\def\iso{\simeq}

\newcommand\abs[1]{\mathord{\left\lvert#1\right\rvert}}
\newcommand\norm[1]{\mathord{\left\Vert#1\right\Vert}}

\newcommand{\rel}{\sim}
\newcommand{\de}{\;\mathrm{d}}
\newcommand{\undefined}{\Diamond}

\DeclareMathOperator{\homI}{h}

\DeclareMathOperator{\interior}{int}

\DeclareMathOperator{\cl}{cl}

\DeclareMathOperator{\Def}{\mathcal{D}}

\DeclareMathOperator{\Con}{\mathcal{C}}

\DeclareMathOperator{\Hom}{\mathrm{H}}
\DeclareMathOperator{\dirlim}{\mathrm{dirlim}}

\newcommand{\Hull}{\mathcal{H}}

\newcommand{\rep}[1]{\left[#1\right]}

\newtheorem{theorem}{Theorem}[section]
\newtheorem{lemma}[theorem]{Lemma}

\theoremstyle{definition}
\newtheorem{definition}[theorem]{Definition}

\theoremstyle{remark}
\newtheorem{remark}[theorem]{Remark}
\newtheorem{example}[theorem]{Example}

\begin{document}
	\title[Nonautonomous Conley Index Theory: The Connecting Homomorphism]{
		Nonautonomous Conley Index Theory\\
		The Connecting Homomorphism}
	
	\author[A. J\"anig]{Axel J\"anig}
	
	\address
	{\textsc{Axel J\"anig}\\
		Institut f\"ur Mathematik\\
		Universit\"at Rostock\\
		18051 Rostock, Germany}
	
	\email{axel.jaenig@uni-rostock.de}

	
	
%
%
	
	\begin{abstract}%
		Attractor-repeller decompositions of isolated invariant sets give rise
		to so-called connecting homomorphisms. These homomorphisms reveal information
		on the existence and structure of connecting trajectories of the underlying
		dynamical system.
		
		To give a meaningful generalization of this general principle to nonautonomous
		problems, the nonautonomous homology Conley index is expressed as a direct limit.
		Moreover, it is shown that a nontrivial connecting homomorphism implies, on the dynamical systems level,
		a sort of uniform connectedness of the attractor-repeller decomposition.
	\end{abstract}
	
	\maketitle
	 
	In previous works \cite{article_naci, article_naci_1} the author developed a nonautonomous Conley index theory.
	The index relies on the interplay between a skew-product semiflow and a nonautonomous
	evolution operator. It can be applied to various nonautonomous problems, including ordinary differential
	equations and semilinear parabolic equations (see \cite{article_naci}).
	
	Every attractor-repeller decomposition
	of an isolated invariant set gives rise to a long exact sequence involving the homology Conley index.
	The connecting homomorphism of this sequence contains information on the connections between repeller
	and attractor. In particular, the connecting homomorphism vanishes if a connecting orbit does not exist.
	
	However, the mere existence of a connecting orbit is a very weak result in
	the context of nonautonomous dynamical systems. This is due to the fact that, 
	in general, a nonautonomous Conley index is always tied to
	a family of nonautonomous dynamical systems. Therefore, for a specific dynamical system, only (in an appropriate sense) uniform properties are meaningful. It will be shown that the existence of connections due to a
	non-trivial connecting homomorphism is such a uniform property.
	
	More precisely, the nonautonomous homology
	Conley index is expressed as a direct limit, which resembles the definition of a discrete Conley index.
	Using the direct limit formula, a notion of uniform connectedness of the attractor-repeller decomposition
	respectively the invariant set is introduced. Moreover, it is shown that a non trivial connecting homomorphism
	implies not only the existence of a connection but a uniform connection of repeller and attractor.
	
	Theorem \ref{th:171017-1012} below, included mainly for illustrative purposes,
	translates the results of this paper to the problem of small perturbations
	of a class of semilinear parabolic equations e.g., reaction diffusion equations.
	The main part of the paper starts with a Preliminaries section, where we collect
	important definitions and results from other works.
	In Section 3, the direct limit formulation is formulated and proved. The notion
	of uniform connectedness is introduced in Section 4. 
	Subsequently, we prove Theorem \ref{th:141205-1538} stating that
	a nontrivial connecting homomorphism implies a uniform connection of
	repeller and attractor.

\begin{section}{Nonautonomous $C^0$-small perturbations of (autonomous) semilinear parabolic equations}
	\label{sec:160219-1750}
	\index{perturbations of semilinear parabolic equations}
	Let $X$ be a Banach space and $A_0$ be a sectorial operator defined on a dense
	subset $\Def(A_0)\subset X$. We are interested in mild solutions of
	\begin{equation}
	\label{eq:150430-1526}
	u_t + A_0u = \hat f(t,u),
	\end{equation}
	which happen to be strong solutions due to regularity assumptions.
	Let us further assume that $A_0$ has compact resolvent.
	
	As often, the operator $A_0$ is assumed to be positive, so
	there is a family of fractional power spaces $X^\alpha$ defined by $A_0$.
	The respective norm is given by $\norm{x}_\alpha := \norm{A^\alpha_0 x}_X$.
	A typical example would be the Laplace-operator on a bounded domain with
	smooth boundary under appropriate boundary conditions (see e.g. \cite{henry, sellyou, pazy}).
	
	We will shortly introduce another metric space $Y$. With every $f\in Y$ there
	is an associated mapping $\hat f$, which serves a a parameter
	for the evolution operator defined by \eqref{eq:150430-1526}. A typical
	example for $\hat f$ is assigning the Nemitskii operator associated
	with a function $f$.
	
	We are interested $C^0$-small nonautonomous perturbations of autonomous
	equations that is,
	\begin{equation*}
		u_t + A_0u = \hat f(u) + \hat g(t,u)
	\end{equation*}
	where $g$ is assumed to be small in an appropriate metric. Note that
	this is much stronger than 
	\begin{equation*}
	u_t + A_0u = \hat f(u) + \eps \hat g(t,u),
	\end{equation*}
	where $\eps$ is assumed to be small.
	
	The main result is the persistence of Morse-decompositions
	and certain solutions: Morse-sets with a non-zero index as well
	as connecting orbits with a non-vanishing connecting homomorphism.
	A typical Morse-set with non-zero index might be a hyperbolic equilibrium
	and a typical Morse-set with a non-vanishing connecting homomorphism
	might be a transversal heteroclinic solution (see \cite{article_index}).
	
	In the sequel, a specific choice for $Y$ is made. Additional material can be found in 
	\cite{article_naci} and \cite{sell1971topological}.
	
	Let $\Omega\subset \IR^N$ be a smooth bounded domain, and suppose that,
	for some $\alpha\in\left[0,1\right[$, there is a continuous inclusion $X^\alpha\subset C(\bar\Omega)$.
	Let $Y$ denote the set of all continuous functions $f:\;\IR\times \bar\Omega\times\IR\to\IR$
	which are subject of the following restriction:
	\begin{enumerate}
		\item[] For some $\delta>0$ and every $C_1>0$, there are constants $C_2=C_2(C_1)$ and $C_3 = C_3(C_1)$ such that
		for all $(t,x,u)\in \IR\times\Omega\times \IR$ with $\abs{u}\leq C_1$
		\begin{equation*}
		\abs{f(t,x,u)} \leq C_2
		\end{equation*}
		and for all $(t_1,x_1,u_1), (t_2,x_2,u_2)\in\IR\times\Omega\times\IR$ with $\abs{u_1},\abs{u_1}\leq C_1$
		\begin{equation*}
		\abs{f(t_1,x_1,u_1) - f(t_2,x_2,u_2)} \leq C_3\left( \abs{t_1-t_2}^\delta + \abs{x_1-x_2}^\delta +  \abs{u_1 - u_2}\right)
		\end{equation*}
	\end{enumerate}
	Defining addition and scalar multiplication pointwise as usual, $Y$ 
	becomes a linear space. We consider a family $(\delta_n)_{n\in\IN}$
	of seminorms 
	\begin{equation*}
	\delta_n(f) := \sup\{ \abs{f(t,x,u)}:\; (t,x,u)\in\IR\times\Omega\times\IR\text{ with }\abs{t},\abs{u}\leq n\}.
	\end{equation*}
	These seminorms give rise to an invariant metric $d$ on $Y$:
	\begin{equation*}
	d(f_1, f_2) := \sum^{\infty}_{n=1} 2^{-n} \frac{\delta_n(f_1-f_2)}{1+\delta_n(f_1-f_2)}.
	\end{equation*}
	The metric $d$ induces the compact-open topology on $Y$, so a sequence of functions
	converges with respect to $d$ if and only if it converges uniformly on bounded subsets
	of $\IR\times\Omega\times\IR$.
	
	To formulate the theorem below, a uniform distance is also required i.e.,
	\begin{equation*}
	d_\text{unif}(f_1,f_2) := \sup_{(t,x,u)\in\IR\times\Omega\times\IR} \abs{f_1(t,x,u) - f_2(t,x,u)}
	\end{equation*}
	Denote
	\begin{equation*}
	h:\;\IR\to \IR\quad h(t) := 
	\begin{cases}
	(t+1)\sin\ln (t+1) & t> 0\\
	0 & t\leq 0
	\end{cases}
	\end{equation*}
	and $t_n := e^{2\pi n} - 1$, $n\in\IN$. Then, $\ln(t_n +s) - 2\pi n\to 0$ as $n\to\infty$
	uniformly for $s$ lying in bounded subsets of $\IR$. Hence, one has 
	\begin{equation*}
	h(t_n + s) = h(t_n) + \int\limits^{t_n+s}_{t_n} \sin\ln (t+1) + \cos\ln (t+1) \de t
	\end{equation*}
	so that $h(t_n + s) \to s$ as $n\to\infty$ uniformly on bounded
	sets.
	
	We are interested in full (defined on $\IR$) solutions of
	a perturbed equation. Suppose that $f\in Y$ is the
	parameter associated with the perturbed equation. 
	Computing the index with respect to $f$ would imply the
	loss of all the information contained in $f$ for negative times.
	The index would be determined by the equation's behaviour
	at large times. This restriction can be overcome by
	using the auxiliary function $h$ defined above. It allows
	to embed $f$ into the $\omega$-limes set of a related parameter, namely:
	\begin{equation*}
	f.h := ((t,x,u)\mapsto f(h(t),x,u)).
	\end{equation*}
	
	It is easy to see that $f\in Y$ implies $f.h\in Y$ and, from
	the calculations above, it follows that $(f.h)^{t_n}\to f$ in $Y$,
	that is, uniformly on bounded subsets. 
	
	Combining this approach with the abstract results of this paper
	and previous works on the subject, one obtains the following theorem.
	
	\begin{theorem}
		\label{th:171017-1012}
		Suppose that $f\in Y$ is autonomous, and let
		$K\subset X^\alpha$ be a compact invariant
		set with respect to the evolution operator (semiflow) on $X^\alpha$
		defined by \eqref{eq:150430-1526}. Let $N\subset X^\alpha$
		be a strongly admissible (e.g. bounded) isolated neighbourhood of $K$.
		
		Let $(A,R)$ be an attractor-repeller decomposition
		of $K$, and assume that the associated connecting homomorphism
		$\partial:\; \Hom_*\Con(f,A)\to \Hom_*\Con(f,R)$ defined
		by the homology attractor-repeller sequence does not vanish.
		
		Let $N_A\subset X^\alpha$ (resp. $N_R$) be an isolating neighbourhood
		for $A$ (resp. $R$), and suppose that $N_A\cap N_R = \emptyset$.
		
		Then, there exists an $\eps>0$ such that the following
		holds true for all $f'\in Y$ with $d_\text{unif}(f,f')<\eps$:
		\begin{enumerate}
			\item[(a)] If $u:\;\IR\to N\subset X^\alpha$ is a solution
			of 
			\begin{equation}
			\label{eq:150430-1652}
			u_t + A_0 u = \hat{f'}(t,u),
			\end{equation}
			then either $u(\IR)\subset N_A\cup N_R$ or
			$\alpha(u)\subset N_R$ and $\omega(u)\subset N_A$.
			\item[(b)] There is a solution $u:\;\IR\to N_R$ of \eqref{eq:150430-1652}.
			\item[(c)] There is a solution $u:\;\IR\to N_A$ of \eqref{eq:150430-1652}.
			\item[(d)] Suppose $U_A\subset X^\alpha$ is a neighbourhood of $A$, $U_R\subset X^\alpha$ a neighbourhood
			of $R$ and $U_A\cap U_A=\emptyset$. Then there is a solution $u:\;\IR\to N$ of \eqref{eq:150430-1652}
			such that $\alpha(u)\subset N_R$, $\omega(u)\subset N_A$ and $u(0)\in X^\alpha\setminus (U_A\cup U_R)$.
		\end{enumerate}
	\end{theorem}
	
	Using the same arguments as below, the
	theorem can be generalized to partially ordered Morse decompositions (see \cite{article_naci_2}).
	Moreover, even if the connecting homomorphism is tribal, Morse-decompositions are still preserved under small perturbations, but depending on the Conley indices of attractor and repeller, the existence of solutions
	can no longer be proved\footnote{At least not with these arguments.}.
	
	Note that for a generic reaction diffusion equation, all equilibria are hyperbolic, and their respective
	stable and unstable manifolds intersect transversally \cite{brunpol}. For such a generic reaction-diffusion 
	equation each heteroclinic connection between equilibria of adjacent Morse indices corresponds to a nontrivial connecting homomorphism \cite{article_index}.
	
	\begin{proof}
		First of all, note that $f.h = f$ as $f$ is autonomous. Furthermore, $Y\times N$,
		$Y\times N_A$ and $Y\times N_R$ are isolating neighbourhoods for $(f.h,K)$, $(f.h, A)$
		and $(f.h,R)$ respectively. 
		
		Therefore, (a) follows from Theorem 2.2 in \cite{article_naci_2}. We will now consider
		(the relevant part of) an attractor-repeller sequence containing the non-vanishing 
		connecting homomorphism $\partial$:
		\begin{equation*}
		\xymatrix@1{
			\ar[r] & \Hom_*\Con(f,R) \ar[r]^-{\partial} & \Hom_{*-1}\Con(f,A) \ar[r] &
		}
		\end{equation*}
		Since $\partial\neq 0$, one necessarily has $\Hom_*\Con(f,R)\neq 0$ and
		$\Hom_*\Con(f,A)\neq 0$. By using Theorem 3.4 in \cite{article_naci_2}, one proves that
		for all $f'$ in a neighbourhood of $f$ given by (a), the attractor-repeller sequence
		above extends to a commutative ladder:
		\begin{equation*}
		\xymatrix{
			\ar[r] & \Hom_*\Con(f,R) \ar[r]^-{\partial} \ar[d]^-\iso& \Hom_{*-1}\Con(f,A) \ar[r] \ar[d]^-\iso& \\
			\ar[r] & \Hom_*\Con(f'.h,R') \ar[r]^-{\partial'} & \Hom_{*-1}\Con(f'.h,A') \ar[r] & 
		}
		\end{equation*}
		Here, we set $R' := \Inv(\omega(f'.h)\times N_R)$ and $A' := \Inv (\omega(f'.h)\times N_A)$.
		
		Consequently, in view of Corollary 4.11 in \cite{article_naci} and because $f'\in \omega(f'.h)$,
		(b) and (c) must hold. Finally, claim (d) is a consequence of Theorem \ref{th:141205-1538},
		stating that $K' = \Inv (\omega(f'.h)\times N)$ is uniformly connected.
	\end{proof}
\end{section}

  	\begin{section}{Preliminaries}
  		The section starts with a collection of useful definitions and
  		terminology, mainly from \cite{article_naci} and \cite{article_naci_1}.
  		Thereafter, we review the concept of index pairs and index triples, followed
  		by the nonautonomous homology Conley index and the connecting homomorphism.
	 	
	 	\begin{subsection}{Quotient spaces}
	 		\index{quotient space}
	 		\begin{definition}
	 			\label{df:140606-1652}
	 			Let $X$ be a topological space, and $A,B\subset X$. Denote
	 			\begin{equation*}
	 			A/B := A/R \cup \{A\cap B\},
	 			\end{equation*}
	 			where $A/R$ is the set of equivalence classes
	 			with respect to the relation $R$ on $A$ which is defined by $x R y$ iff $x=y$ or $x,y\in B$.
	 			
	 			We consider $A/B$ as a topological space
	 			endowed with the quotient topology with respect to the
	 			canonical projection $q:\; A\to A/B$, that is, a
	 			set $U\subset A/B$ is open if and only if
	 			\begin{equation*}
	 			q^{-1}(U) = \bigcup_{x\in U} x
	 			\end{equation*}
	 			is open in $A$.
	 		\end{definition}
	 		
	 		Recall that the quotient topology is the final topology
	 		with respect to the projection $q$. 
	 		
	 		\begin{remark}
	 			The above definition is compatible
	 			with the definition used in \cite{hib}
	 			or \cite{ryb}. The only difference
	 			occurs in the case $A\cap B=\emptyset$, where
	 			we add $\emptyset$, which is never an equivalence class,
	 			instead of an arbitrary point.
	 		\end{remark}
	 	\end{subsection}
	 	
	 	\begin{subsection}{Evolution operators and semiflows}
	 		Let $X$ be a metric space. Assuming that $\undefined\not\in X$, we
	 		introduce a symbol $\undefined$, which means "undefined". The intention is
	 		to avoid the distinction if an evolution operator is defined for a given argument
	 		or not. Define $\overline A := A\dot\cup\{\undefined\}$ whenever $A$ is a set with $\undefined\not\in A$.
	 		Note that $\overline{A}$ is merely a set, the notation does not contain any
	 		implicit assumption on the topology. 
	 		
	 		\index{evolution operator}
	 		\index{semiflow}
	 		
	 		\begin{definition}
	 			\label{df:131009-1504}
	 			\index{evolution operator}
	 			Let $\Delta:=\{(t,t_0)\in\IR^+\times\IR^+:\; t\geq t_0\}$.
	 			A mapping $\Phi:\; \Delta\times \overline{X}\to \overline{X}$
	 			is called an {\em evolution operator} if
	 			\begin{enumerate}
	 				\item $\Def(\Phi) := \{(t,t_0,x)\in\Delta\times X:\; \Phi(t,t_0,x)\neq\undefined\}$ is open
	 				in $\IR^+\times\IR^+\times X$;
	 				\item $\Phi$ is continuous on $\Def(\Phi)$;
	 				\item $\Phi(t_0,t_0, x) = x$ for all $(t_0,x)\in\IR^+\times X$;
	 				\item $\Phi(t_2 ,t_0,x) = \Phi(t_2,t_1,\Phi(t_1,t_0,x))$ for all $t_0\leq t_1\leq t_2$ in $\IR^+$ and $x\in X$;
	 				\item $\Phi(t,t_0,\undefined) = \undefined$ for all $t\geq t_0$ in $\IR^+$.
	 			\end{enumerate}
	 			
	 			A mapping $\pi:\; \IR^+\times \overline{X}\to \overline{X}$ is called {\em semiflow}
	 			if $\tilde\Phi(t+t_0,t_0,x) := \pi(t,x)$ defines an evolution operator.
	 			To every evolution operator $\Phi$, there is an associated
	 			(skew-product) semiflow $\pi$ on an extended phase space $\IR^+\times X$, defined by $(t_0,x)\pi t = (t_0+t, \Phi(t+t_0,t_0,x))$.
	 			
	 			A function $u:\; I\to X$ defined on a subinterval $I$ of $\IR$ is called a {\em solution of (with respect to) $\Phi$}
	 			if $u(t_1) = \Phi(t_1,t_0,u(t_0))$ for all $\left[t_0,t_1\right]\subset I$.
	 		\end{definition}
	 		
	 		
	 		\begin{definition}
	 			Let $X$ be a metric space, $N\subset X$ and $\pi$ a semiflow on $X$. The set
	 			\begin{equation*}
	 			\Inv^-_\pi(N) := \{x\in N:\;\text{ there is a solution }u:\;\IR^-\to N\text{ with }u(0) = x\}
	 			\end{equation*}
	 			is called the {\em largest negatively invariant subset of $N$}.
	 			
	 			The set
	 			\begin{equation*}
	 			\Inv^+_\pi(N) := \{x\in N:\;x\pi\IR^+\subset N\}
	 			\end{equation*}
	 			is called the {\em largest positively invariant subset of $N$}.
	 			
	 			The set
	 			\begin{equation*}
	 			\Inv_\pi(N) := \{x\in N:\;\text{ there is a solution }u:\;\IR\to N\text{ with }u(0) = x\}
	 			\end{equation*}
	 			is called the {\em largest invariant subset of $N$}.
	 		\end{definition}

	 		\index{skew-product semiflow}
	 		\index{cocycle mapping}
	 		
	 		Let $X$ and $Y$ be metric spaces, 
	 		and assume that $y\mapsto y^t$ is
	 		a global\footnote{defined for all $t\in\IR^+$} semiflow 
	 		on $Y$, to which we will refer as $t$-translation.
	 		
	 		\begin{example}
	 			Let $Z$ be a metric space, and let $Y := C(\IR^+, Z)$
	 			be a metric space such that a sequence of functions converges
	 			if and only if it converges uniformly on bounded sets.
	 			The translation can now be defined canonically by
	 			$y^t(s) := y(t+s)$ for $s,t\in\IR^+$.
	 		\end{example}
	 		
	 		A suitable abstraction of many non-autonomous problems is given
	 		by the concept of skew-product semiflows introduced below.
	 		
	 		\begin{definition}
	 			We say that $\pi = (.^t,\Phi)$ is a skew-product semiflow on $Y\times X$
	 			if $\Phi:\; \IR^+\times \overline{Y\times X}\to \overline{Y\times X}$ is a mapping such that 
	 			\begin{equation*}
	 			(t,y,x)\pi t := 
	 			\begin{cases}
	 			(y^t, \Phi(t,y,x)) & \Phi(t,y,x)\neq\undefined\\
	 			\undefined & \text{otherwise}
	 			\end{cases}
	 			\end{equation*}
	 			is a semiflow on $Y\times X$.
	 		\end{definition}

			A skew-product semiflow gives rise to evolution operators.
	 		
	 		\begin{definition}
	 			Let $\pi = (.^t, \Phi)$ be a skew-product semiflow and $y\in Y$. Define
	 			\begin{equation*}
	 			\Phi_y(t+t_0,t_0,x) := \Phi(t,y^{t_0},x).
	 			\end{equation*}
	 			It is easily proved that $\Phi_y$
	 			is an evolution operator in the sense of Definition \ref{df:131009-1504}.
	 		\end{definition}
	 		
	 		\begin{definition}
	 			\index{positive hull}
	 			\index{$\Hull^+(y_0)$}
	 			\index{$Y_c$}
	 			For $y\in Y$ let
	 			\begin{equation*}
	 			\Hull^+(y) := \cl_Y \{y^t:\;t\in\IR^+\} 
	 			\end{equation*}
	 			denote the positive hull of $y$. Let $Y_c$ denote the set of 
	 			all $y\in Y$ for which $\Hull^+(y)$ is compact.
	 		\end{definition}

			\begin{definition}
				\label{df:isolating-neighborhood-1}
				Let $y_0\in Y$ and $K\subset \Hull^+(y_0)\times X$ be an invariant set. 
				A closed set $N\subset Y\times X$ (resp. $N\subset \Hull^+(y_0)\times X$)
				is called an {\em isolating neighbourhood} for $(y_0, K)$ (in $Y\times X$) (resp. in $\Hull^+(y_0)\times X$)
				provided that:
				\begin{enumerate}
					\item $K\subset \Hull^+(y_0) \times X$
					\item $K\subset \interior_{Y\times X} N$ (resp. $K\subset \interior_{\Hull^+(y_0)\times X} N$)
					\item $K$ is the largest invariant subset of $N\cap (\Hull^+(y_0)\times X)$
				\end{enumerate}				 	 							
			\end{definition}				
%
	 		
	 		The following definition is a consequence of the slightly modified notion
	 		of a semiflow (Definition \ref{df:131009-1504}) but not a semantic change
	 		compared to \cite{hib}, for instance.
	 		
	 		\begin{definition}
	 			\index{does not explode}
	 			We say that $\pi$ explodes in $N\subset Y\times X$
	 			if $x\pi\left[0,t\right[\subset N$ and $x\pi t=\undefined$.
	 		\end{definition}
	 		
	 		Following \cite{ryb}, we formulate the following
	 		asymptotic compactness condition.
	 		
	 		\begin{definition}
	 			A set $M\subset Y\times X$ is called {\em strongly admissible} provided
	 			the following holds:
	 			
	 			Whenever $(y_n,x_n)$ is a sequence in $M$ and $(t_n)_n$ is a sequence
	 			in $\IR^+$ such that $(y_n,x_n)\pi\left[0,t_n\right]\subset M$, then
	 			the sequence $(y_n,x_n)\pi t_n$ has a convergent subsequence.
	 		\end{definition}
	 		
%
%
	 	\end{subsection}
 			
  		\begin{subsection}{Index pairs and index triples}
  			The notion of (basic) index pairs relies on \cite{article_naci}
  			and was introduced in \cite{article_naci_1}.
  			
		 	\begin{definition}
        		\label{df:basic_index_pair}
	  	 		A pair $(N_1, N_2)$ is called a {\em (basic) index pair} relative to a
	  	 		semiflow $\chi$ in $\IR^+\times X$ if
	  	 		\begin{enumerate}
	  	 			\item[(IP1)] $N_2\subset N_1\subset \IR^+\times X$, $N_1$ and $N_2$ are
	  	 			closed in $\IR^+\times X$
	  	 			\item[(IP2)] If $x\in N_1$ and $x\chi t\not\in N_1$ for some $t\in\IR^+$, 
	  	 			then $x\chi s\in N_2$ for some $s\in\left[0,t\right]$;
	  	 			\item[(IP3)] If $x\in N_2$ and $x\chi t\not\in N_2$ for some $t\in\IR^+$, then
	  	 			$x\chi s\in (\IR^+ \times X)\setminus N_1$ for some $s\in\left[0,t\right]$.
	  	 		\end{enumerate}
	  	 	\end{definition}
	  			 	
		 	\begin{definition}
		 		\label{df:140128-1503}
		 		Let $y_0\in Y$ and $(N_1, N_2)$ be a basic index pair in $\IR^+\times X$ relative to $\chi_{y_0}$. Define
		 		$r:=r_{y_0}:\; \IR^+\times X\to \Hull^+(y_0)\times X$ by $r_{y_0}(t,x) := (y^t_0,x)$.
	  			 		
		 		Let $K\subset \omega(y_0)\times X$ be an (isolated) invariant
		 		set. We say that $(N_1, N_2)$ is a (strongly admissible) index pair\footnote{Every index pair in the
	 			sense of Definition \ref{df:140128-1503} is assumed to be strongly admissible.} for $(y_0, K)$ if:
		 		\begin{enumerate}
		 			\item[(IP4)] there is a strongly admissible isolating neighborhood $N$ of $K$ in $\Hull^+(y_0)\times X$
			 			such that $N_1\setminus N_2 \subset r^{-1}(N)$;
		 			\item[(IP5)] there is a neighbourhood $W$ of $K$ in $\Hull^+(y_0)\times X$
 			 			such that $r^{-1}(W)\subset N_1\setminus N_2$.
		 		\end{enumerate}
		 	\end{definition}
		 	
 		 	\begin{definition}
 		 		Let $(N_1, N_2)$ be an index pair in $\IR^+\times X$ (relative to the semiflow
 		 		$\chi$ on $\IR^+\times X$). For $T\in\IR^+$, we set
 		 		\begin{equation*}
	   		 		N^{-T}_2 := N^{-T}_2(N_1) := \{(t,x)\in N_1:\; \exists s\leq T\; (t,x)\chi s\in N_2\}.
 		 		\end{equation*}
 		 	\end{definition}
		 		 	
			To define index triples, the notion of an attractor-repeller decomposition is 
			required. 
			
	 		First of all, $\alpha$ and $\omega$-limes sets
	 		can be defined as usual.
	 		\begin{align*}
	 		\alpha(u) &:= \bigcap_{t\in\IR^-} \cl_{\Hull^+(y_0)\times X} u(\left]-\infty,t\right])\\
	 		\omega(u) &:= \bigcap_{t\in\IR^+} \cl_{\Hull^+(y_0)\times X} u(\left[t,\infty\right[)
	 		\end{align*}
				 		
	 		Based on these definitions, the notion of an attractor-repeller decomposition
	 		can be made precise.
				 		
	 		\begin{definition}
	 			Let $y_0\in Y$ and $K\subset \Hull^+(y_0)\times X$ be an isolated
	 			invariant set. $(A,R)$ is an {\em attractor-repeller decomposition} of 
	 			$K$ if $A,R$ are disjoint isolated invariant subsets of $K$ and for every solution $u:\;\IR\to K$
				one of the following alternatives holds true.
				\begin{enumerate}
					\item $u(\IR)\subset A$
					\item $u(\IR)\subset R$
					\item $\alpha(u)\subset R$ and $\omega(u)\subset A$
				\end{enumerate}
				We also say that $(y_0, K, A, R)$ is an attractor-repeller decomposition.
	 		\end{definition}
		 	
			Finally, index triples (which correspond to attractor-repeller decompositions)
			are defined.
	  			 	
	 		\begin{definition}
	 			\label{df:index_triple}
	 			Let $y_0\in Y$ and $K\subset \Hull^+(y_0)\times X$ be an isolated
	 			invariant set admitting a strongly admissible isolating neighbourhood $N$.
  				Suppose that $(A,R)$ is an attractor-repeller decomposition of 
  	 			$K$.
  			 			
  	 			A triple $(N_1, N_2, N_3)$ is called an {\em index triple} for
  				$(y_0,K,A,R)$ provided that:
  	 			\begin{enumerate}
  	 				\item $N_3\subset N_2\subset N_1$
  	 				\item $(N_1, N_3)$ is an index pair for $(y_0,K)$
  	 				\item $(N_2, N_3)$ is an index pair for $(y_0,A)$
  	 			\end{enumerate}
  	 		\end{definition}
  		\end{subsection} 
  		
  		\begin{subsection}{Homology Conley index and attractor repeller decompositions}
  			A {\em connected simple system} is a small category such that given
  			a pair $(A,B)$ of objects, there is exactly one morphism $A\to B$.
  			
			Let $y_0\in Y_c$ and $K\subset \Hull^+(y_0)\times X$
			be an isolated invariant set for which there is a strongly admissible
			isolating neighbourhood.	The categorial Conley index $\Con(y_0, K)$ (as defined in \cite{article_naci_1}) is a
			subcategory of the homotopy category of pointed spaces and a connected simple system. 
			Its objects are the index pairs for $(y_0, K)$. Roughly speaking, one can think
			of an index pair with collapsed exit set as a representative of the index. All of the
			representatives are isomorphic in the homotopy category of pointed spaces.
			
			Let $(\Hom_*,\partial)$ denote a homology theory with compact supports \cite{spanier}.
			Recall that $\Hom_*$ is a covariant function from the category of topological
			pairs to the category of graded abelian groups (or modules). 
			
			Define the {\em homology Conley index} $\Hom_*\Con(y_0,K)$ to be the following
			connected simple system: $\Hom_*(N_1/N_2, \{N_2\})$ is an object
			whenever
			$(N_1/N_2, N_2)$ is an object of $\Con(y_0, K)$. The morphisms of $\Hom_*\Con(y_0, K)$
			are obtained analogously from the morphisms of $\Con(y_0, K)$. 
			Note that we also write $\Hom_*(A,a_0) := \Hom_*(A,\{a_0\})$ provided the meaning is clear.
			
			Let $(y_0, K, A, R)$ be an attractor-repeller decomposition. There is a long exact \cite{article_naci_1}
			sequence
			\begin{equation*}
	 			\xymatrix@1{
				\ar[r] & \Hom_*\Con(y_0,A) \ar[r] & \Hom_*\Con(y_0,K) \ar[r]
 				& \Hom_*\Con(y_0,R) \ar[r]^-{\partial} & \Hom_{*-1}\Con(y_0,A)
 				\ar[r]&
	 			}
			\end{equation*}
			where $\partial$ denotes the connecting homomorphism.
  		\end{subsection}
  		
 \end{section}

\begin{section}{The homology Conley index as a direct limit}
	Let $(N_1, N_2)$ be a basic index pair. Another basic index pair 
	is $(N_1(\left[t,\infty\right[),\linebreak[0] N_2(\left[t,\infty\right[))$, where $t>0$ is arbitrary
	and $N_i(\left[t,\infty\right[) = \{(s,x)\in N_i:\;s\geq t\}$ for $i\in\{1,2\}$. 
	The inclusion $(N_1(\left[t,\infty\right[), N_2(\left[t,\infty\right[))\subset(N_1, N_2)$
	induces a homotopy equivalence between the (pointed) quotient spaces
	$(N_1(\left[t,\infty\right[)/N_2(\left[t,\infty\right[),\linebreak[0] N_2(\left[t,\infty\right[))$
	and $(N_1/N_2, N_2)$.
	Apparently,	only the index pair at large times is relevant. 
	
	In the present section,
	this limit behaviour will be studied.
	Finite sections of an index pair
	$(N_1, N_2)$, that is, sets of the form $N_i\left[\alpha,\beta\right]
	= N_i \cap (\left[\alpha,\beta\right]\times X)$, in conjunction with appropriate morphisms
	form a direct system. The index $\Hom_*(N_1/N_2, N_2/N_2)$ is then proved
	to be isomorphic to a direct limit obtained from these sections. 
	
	It is interesting to note that 
	this result (in particular Lemma \ref{le:150114-1618})
	resembles constructing a Conley index for discrete time dynamical systems
	(see e.g. \cite{ci_discrete}). In this
	paper, however, we will focus on the use of the direct limit representation
	of the index as a tool. 
	
	For the rest of this section, let $\Lambda$ be a set and $\leq$ a partial order on $\Lambda$.
	Recall \cite{spanier} that a {\em direct system of sets} is
	a family $(A_\alpha)_{\alpha\in \Lambda}$ of sets and a family of functions
	$(f_{\alpha,\beta})$, where $\alpha,\beta\in\Lambda$, $\alpha\leq \beta$
	and $f_{\alpha,\beta}:\; A_\alpha\to A_\beta$.
	
	\index{direct limit}
	The {\em direct limit} $\dirlim (A_\alpha, f_{\alpha,\beta})$
	of $(A_\alpha, f_{\alpha,\beta})$ is the
	set of equivalence classes in $\bigcup_{\alpha\in\Lambda} \{\alpha\}\times A_\alpha$
	under the relation $\rel$, which is defined as follows:
	Let $\alpha,\beta\in\Lambda$ and 
	$(a,b)\in A^\alpha\times A^\beta$. $(\alpha,a)\rel (\beta,b)$ if and only if
	there is a $\gamma\in \Lambda$ such that $\alpha,\beta\leq\gamma$
	and $f_{\alpha,\gamma}(a) = f_{\beta,\gamma}(b)$.
	
	Let $(X,d)$ be a complete metric space, and $V\subset \IR^+\times X$.
	We set 
	\begin{equation*}
	\begin{split}
	V(t) &:= \{x:\; (t,x)\in V\}\\
	V(\left[a,b\right]) := V\left[a,b\right] &:= \{(t,x)\in V:\; t\in\left[a,b\right]\}.
	\end{split}
	\end{equation*}
	
	\begin{definition}
		\index{regular index pair}
		\index{exit time (inner)}
		An index pair $(N_1, N_2)$ is called {\em regular} (with respect to $y_0$) if
		the (inner) exit time $T_i:\; N_1\to \left[0,\infty\right]$, $T_i(x) := \sup\{t\in\IR^+:\; x\chi_{y_0}\left[0,t\right]\subset N_1\setminus N_2\}$
		is continuous.
	\end{definition}
	
	The main motivation for regular index pairs are Lemma \ref{le:140113-1642} below
	and Lemma \ref{le:150114-1618} at the end of this section. As
	stated subsequently in Lemma \ref{le:140107-1916},
	it is easily possible to obtain regular index pairs by modifying (enlarging) the exit set appropriately.
	The following notational shortcut is used frequently.
	\begin{equation*}
		\Hom_*[A,B] := \Hom_*(A/B, \{B\})
	\end{equation*}

	\begin{lemma}
		\label{le:140113-1642}
		Let $(N_1, N_2)$ be a regular index pair in $\IR^+\times X$.
		
		Consider
		the direct system $(A_\alpha, f_{\alpha,\beta})$ for $\alpha,\beta\in\Lambda$, where $\Lambda$ denotes the
		set of all non-empty compact subintervals of $\IR^+$
		ordered by inclusion, and $A_\alpha := H_*[N_1(\alpha),N_2(\alpha)]$. For
		$\alpha\subset \beta$, let $i_{\alpha,\beta}:\; (N_1(\alpha), N_2(\alpha))\to (N_1(\beta), N_2(\beta))$
		denote the respective inclusion and set $f_{\alpha,\beta} := \Hom_*(i_{\alpha,\beta}):\; A_\alpha\to A_\beta$.
		
		Then, the inclusions $i_\alpha:\; (N_1(\alpha), N_2(\alpha))\to (N_1, N_2)$ induce
		an isomorphism
		\begin{equation*}
		j: \dirlim (\Hom_*(N_1(\alpha), N_2(\alpha)), f_{\alpha,\beta}) \to \Hom_*[N_1, N_2],\quad \rep{(\alpha,x)} \mapsto \Hom_*(p\circ i_{\alpha})(x),
		\end{equation*}
		where $p:\; N_1\to N_1/N_2$ denotes the canonical projection.
	\end{lemma}
	
	\begin{lemma}
		\label{le:140107-1916}
		Let $(N_1, N_2)$ be an index pair for $(y_0, K)$. Then there are a constant $\tau\in\IR^+$
		and a set $N'_2\subset N_1$ such that:
		\begin{enumerate}
			\item $N_2\subset N'_2\subset N^{-\tau}_2$
			\item $(N_1, N'_2)$ is a regular index pair for $(y_0, K)$.
		\end{enumerate}
	\end{lemma}
	
	\begin{proof}
		By Lemma 4.8 in \cite{article_naci_1}, $N^{-T}_2$ is a neighbourhood
		of $N_2$ in $N_1$ provided that $T$ is sufficiently large.    
		It follows that $N_2\cap \cl (N_1\setminus N^{-T}_2) = \emptyset$.
		By Urysohn's lemma, there exists a continuous function $f:\;N_1\to \left[0,1\right]$
		such that $f(x) = 0$ on $N_2$ and $f(x) = 1$ on $\cl (N_1\setminus N^{-T}_2)$.
		
		Set
		\begin{equation*}
		\lambda(x) := \int\limits^{T(x)}_0 f(x\chi_{y_0} s)\de s,
		\end{equation*}
		where $T(x) := \sup \{ t\in\IR^+:\; x\chi_{y_0} \left[0,t\right]\subset N_1\setminus N_2\}$, in order to
		guarantee that the integrand is defined.
		
		It is easy to see that $\lambda(x) = 0$ on $N_2$ and $\lambda(x)\leq T(x)$ for all $x\in N_1$.
		Next, we are going to prove the left-hand inequality of
		\begin{equation}
		\label{eq:150921-1903}
		T(x) - T \leq \lambda(x) \leq T(x).
		\end{equation}
		One has $\lambda(x)\geq 0$ for all $x\in N_1$,
		so let $x\in N_1$ with $T(x)>T$. It follows that $f(x\chi_{y_0} s) = 1$ for all
		$s\in\left[0,T(x)-T\right]$, so 
		\begin{equation*}
		T(x) - T = \int\limits^{T(x)-T}_0 f(x\chi_{y_0} s)\,ds \leq \lambda(x),
		\end{equation*}
		proving \eqref{eq:150921-1903} for all $x\in N_1$.
		
		We need to show that $\lambda$ is continuous.
		Suppose that $x_n\to x_0$ is a sequence and $\lambda(x_0)<\infty$.
		Assume additionally that $T(x_n)$ is unbounded,
		so it is possible extract a subsequence $x'_n$ with $T(x'_n)\to\infty$. We have
		$x'_n\chi_{y_0} s\in N_1\setminus N^{-T}_2$ for all $s< T(x'_n)-T$ and all $n\in\IN$,
		so $x_0\chi_{y_0} s\in \cl (N_1\setminus N^{-T}_2) \subset N_1\setminus N_2$ for all $s\in\IR^+$, which in turn
		implies that $T(x_0) = \infty$. However by \eqref{eq:150921-1903}, $\lambda(x_0) \geq T(x_0) - T =\infty$,
		which is a contradiction. Consequently, the sequence $(T(x_n))_n$ must be bounded.
		
		We further have $x_n\chi_{y_0} \left[0,T(x_n)\right]\subset N_1$
		and $x_n\chi_{y_0} T(x_n)\in N_2$ for all $n\in\IN$. $T(x_n)$ is bounded, so
		we may choose a subsequence $x'_n$ with $T(x'_n)\to t_0<\infty$. It follows
		that 
		\begin{equation*}
		\lambda(x'_n)\to \int\limits^{t_0}_0 f(x\chi_{y_0} s)\,ds = \lambda(x_0),
		\end{equation*}
		where the last equality stems from the facts that $N_2$ is positively invariant
		and $f(x)=0$ on $N_2$.
		This readily implies that $\lambda(x_n)\to \lambda(x_0)$.
		
		Finally if $\lambda(x_0) = \infty$, then $x_0\chi_{y_0} s\in N_1\setminus N^{-T}_2$ for all
		$s\in\IR^+$. Arguing by contradiction, assume that there exist a real number $t_0$ and
		a subsequence $x'_n$
		with $\lambda(x'_n)\leq t_0$ for all $n\in\IN$. From \eqref{eq:150921-1903}, one obtains
		that $x'_n\chi_{y_0} t_n \in N^{-T}_2$
		for some $t_n\in\left[0,t_0\right]$. Taking subsequences, we may assume w.l.o.g.
		that $t_n\to t'_0\leq t_0$, so $x_0\chi_{y_0} t'_0\in N^{-T}_2$, implying that
		$\lambda(x_0)\leq T(x_0) \leq t'_0 + T$. This is a contradiction and completes
		the proof that $\lambda$ is continuous.
		
		It is easy to see 
		that $N'_2:=\lambda^{-1}(\left[0,T+1\right])$ is a closed
		neighborhood of $N^{-T}_2$ in $N_1$. Moreover, $\lambda$ is monotone decreasing along the semiflow, so $(N_1, N'_2)$
		is an index pair. 
		
		By \eqref{eq:150921-1903}, it holds that $N'_2\subset N^{-\tau}_2$, where $\tau := 2T + 1$. It follows
		from Lemma 2.7 in \cite{article_naci_1} that $(N_1, N^{-\tau}_2)$ is an
		index pair. In conjunction with Lemma 2.8 in \cite{article_naci}, one concludes that
		$(N_1, N'_2)$ is an index pair for $(y_0, K)$.
		
		Let $x\in N_1\setminus N'_2$ and recall the definition $T_i(x) := \sup\{t\in\IR^+:\; x\chi_{y_0} \left[0,t\right]\subset N_1\setminus N'_2\}$
		of the inner exit time. We have $\lambda(x\chi_{y_0} T_i(x)) = T+1$ and
		$f(x) = 1$ on $N_1\setminus N'_2$, so $\lambda(x) = T_i(x) + T + 1$.
		$\lambda$ is continuous as already proved, so $(N_1, N'_2)$
		is a regular index pair as claimed.
	\end{proof}
	
	Using regular index pairs, it is easy to prove the following
	stronger version of Corollary 4.9 in \cite{article_naci}.
	
	\begin{lemma}
		\label{le:detector2}
		Let $y_0\in Y$ and $K\subset \Hull^+(y_0)\times X$ an
		isolated invariant set admitting a strongly admissible isolating neighbourhood.
		
		If $(N_1, N_2)$ is an index pair for $(y_0, K)$, and $\homI(N_1/N_2, N_2)\neq \bar 0$,
		then there are a $t_0\in\IR^+$ and a solution $u:\;\left[t_0,\infty\right[\to N_1\setminus N_2$
		of $\Phi_{y_0}$.
	\end{lemma}
	
	\begin{proof}
		In view of Lemma \ref{le:140107-1916}, one may assume without loss of generality
		that $(N_1, N_2)$ is a regular index pair. Suppose that $(N_1, N_2)$ is such that
		for every $t_0\in\IR^+$ there does not exist a solution $u:\;\left[t_0,\infty\right[\to N_1\setminus N_2$
		of $\Phi_{y_0}$. Then the (continuous) exit time $T_i$ satisfies
		\begin{equation*}
		T_i(x) = \sup_{t\in\IR^+} \{ x\chi_{y_0} \left[0,t\right]\subset N_1\setminus N_2 \} < \infty
		\text{ for all }x\in N_1.
		\end{equation*}  
		
		It is easy to see that for each $x\in N_1$, $x\chi_{y_0}\left[0,T_i(x)\right]\subset N_1$ and
		$x\chi_{y_0} T_i(x)\in N_2$. One can define $H:\;\left[0,1\right]\times N_1 \to N_1$ by
		\begin{equation*}
		H(\lambda, x) := x\chi_{y_0} (\lambda T_i(x)).
		\end{equation*}
		$H$ is continuous, and $H(\lambda,x) = x$ for all $(\lambda, x)\in\left[0,1\right]\times N_2$.
		Consequently, $(N_1/N_2, N_2)$ and $(N_2/N_2, N_2)$ are homotopy equivalent, completing
		the proof because $\homI(N_2/N_2, N_2) = \bar 0$.
	\end{proof}
	
	\begin{lemma}
		\label{le:141210-1901}
		Let $(N_1, N_2)$ be a regular index pair with respect to $y_0\in Y$. Then
		the projection $p:\; N_1\to N_1/N_2$ induces an isomorphism 
		$p_*:\; \Hom_*(N_1, N_2)\to \Hom_*(N_1/N_2, N_2/N_2)$.
	\end{lemma}
	
	\begin{proof}
		The (inner) exit time $T_i(x):=\sup\{t\in\IR^+:\; x\chi_{y_0}\left[0,t\right] \subset N_1\setminus N_2\}$
		is continuous. Therefore, $N'_2 := N^{-1}_2 = \{x\in N_1:\; x\chi_{y_0} s\in N_2\text{ for some }
		s\in\left[0,1\right]\}$ is a neighbourhood of $N_2$ in $N_1$.
		Define $H:\; \left[0,1\right]\times N_1 \to N_1$ by
		$H(\lambda, x) := x\chi_{y_0} (\lambda \min\{ T_i(x), 1\})$. Using $H$, we conclude
		that there are inclusion induced isomorphisms
		\begin{align*}
		&\Hom_*(N_1, N_2) \to \Hom_*(N_1, N'_2) \\
		&\Hom_*(N_1/N_2, \{N_2\}) \to \Hom_*(N_1/N_2, N'_2/N_2).
		\end{align*}
		
		Using the excision property of homology, 
		it follows that $p$ induces an isomorphism $\Hom_*(N_1, N'_2) \to \Hom_*(N_1/N_2, N'_2/N_2)$.
		The proof is complete.
	\end{proof}
	
	\begin{proof}[of Lemma \ref{le:140113-1642}]
		In view of Lemma \ref{le:141210-1901}, it is sufficient
		to consider the inclusion induced mapping
		\begin{equation*}
		j': \dirlim (\Hom_*(N_1(\alpha), N_2(\alpha)), \Hom_*(i_{\alpha,\beta})) \to \Hom_*(N_1, N_2).
		\end{equation*}
		$j'$ is an isomorphism since $\Hom$ is assumed to be
		a homology theory with compact supports (see e.g. \cite[Theorem 13 in Section 4.8]{spanier}).
	\end{proof}
	
	
	
	\begin{lemma}
		\label{le:140114-1611}
		Let the direct system $(A_\alpha, f_{\alpha,\beta})$ be defined as
		in Lemma \ref{le:140113-1642}, $a<c$, and $\alpha := [a,c]\subset [b,c] =: \beta$.
		
		Then, $f_{\alpha, \beta}$ is an isomorphism.
	\end{lemma}
	
	\begin{proof}
		Let $h>0$ and $\gamma:=\left[d-h,d\right]\subset \IR^+$ be an otherwise
		arbitrary interval. Since $(N_1, N_2)$ is assumed to be a regular index pair,
		the inner exit time $T(x) := \sup\{t\in\IR^+:\; x\chi_{y_0}\left[0,t\right]\subset N_1\setminus N_2\}$ is
		continuous. 
		
		We can define a continuous mapping $H:\; \left[0,1\right]\times N_1(\gamma)\to N_1(\gamma)$ by
		\begin{equation*}
		H(\lambda, (d-t,x)) := \begin{cases}
		(d-t,x)\chi_{y_0} (\lambda t) & \lambda t\leq T(x) \\
		(d-t,x)\chi_{y_0} T(x) & \lambda t> T(x).
		\end{cases}
		\end{equation*}
		It follows that $(N'_1(\gamma), N_2(\gamma))$ is a strong deformation retract
		of $(N_1(\gamma), N_2(\gamma))$, where we set
		\begin{equation*}
		N'_1 := N_1(\{d\}) \cup N_2.
		\end{equation*}
		
		Therefore, the inclusion $(N'_1(\gamma), N_2(\gamma))\subset (N_1(\gamma), N_2(\gamma))$
		defines an isomorphism
		\begin{equation*}
		\Hom_*(N'_1(\gamma), N_2(\gamma)) \to \Hom_*(N_1(\gamma), N_2(\gamma)).
		\end{equation*}
		Moreover, the inclusion $(N'_1(\alpha), N_2(\alpha)) \to (N'_1(\beta), N_2(\beta))$
		induces an isomorphism in homology by excision\footnote{Here, the assumption $a\neq c$ is used.}. 
		
		Summing up, there is a commutative diagram, where every arrow denotes an isomorphism
		induced by the inclusion of the respective subspaces:
		\vspace{5mm}
		\begin{equation*}
		\xymatrix@1@C=3mm{
			\Hom_*(N_1(\alpha), N_2(\alpha)) \ar@/^2pc/[rrr]^-{f_{\alpha,\beta}}&
			\Hom_*(N'_1(\alpha), N_2(\alpha)) \ar[l] \ar[r]&
			\Hom_*(N'_1(\beta), N_2(\beta)) \ar[r] &
			\Hom_*(N_1(\beta), N_2(\beta))
		}
		\end{equation*}
	\end{proof}
	
	\begin{lemma}
		\label{le:150306-1820}
		Let $y_0\in Y$, $(N_1, N_2)$ be a regular index pair with respect to $y_0$ and
		$\eps>0$ be arbitrary (not necessarily small). Let $\Gamma$ denote
		the set of all non-empty compact subintervals of $\IR^+$ ordered by inclusion. Let 
		$\alpha,\beta\in\Gamma$, $N^{-\eps}_2 := N^{-\eps}_2(N_1)$ and
		$A^\eps_\alpha := \Hom_*(N_1(\alpha),N^{-\eps}_2(\alpha))$. Finally, let $f^\eps_{\alpha,\beta}:\;A^\eps_\alpha\to A^\eps_\beta$
		be inclusion induced. 
		
		Then for every\footnote{In contrast to Lemma \ref{le:140114-1611}, the case $\alpha = \{c\}$ is included.}
		pair $\left[a,c\right]\subset\left[b,c\right]$ of subintervals of $\IR^+$, it
		holds that $f^\eps_{\alpha, \beta}$ is an isomorphism.
	\end{lemma}
	
	\begin{proof}
		As in the proof of Lemma \ref{le:140114-1611}, it can be shown that
		$(N'_1, N'_2) := (N_1(\{c\})\cup N_2(\left[a,c\right]), N^{-\eps}_2(\{c\})\cup N_2(\left[a,c\right])$ is a strong deformation retract
		of $(N_1\left[a,c\right], N^{-\eps}_2(\left[a,c\right]))$. 
		
		We have $N^{-\eps}_2 = T^{-1}(\left[0,\eps\right])$, where $T:\;N_1\to\left[0,\infty\right]$ denotes the
		inner exit time $T(x) := \sup \{t\in\IR^+:\; x\chi_{y_0}\left[0,t\right]\subset N_1\setminus N_2\}$. Due to the continuity of $T$,
		$N^{-\eps/2}_2(N_1)\cap N'_1$ is a neighbourhood of $N_2$ in $N'_1$. Hence, $(N_1(\{c\}), N^{-\eps}_2(\{c\}))\subset (N'_1, N'_2)$ induces
		an isomorphism in homology by excision.
		
		Further details are omitted.
	\end{proof}
	
	\begin{lemma}
		\label{le:150114-1618}
		Let $(N_1, N_2)$ be a regular index pair with respect to $y_0\in Y$, $\eps>0$, and let the direct system $(A^\eps_\alpha, f_{\alpha,\beta})$
		be defined as in Lemma \ref{le:150306-1820}.
		
		Suppose we are given a strictly monotone increasing sequence $(a_n)_n$ in $\IR^+$ with $a_n\to\infty$. 
		Define a direct system $(B^\eps_k, g_{k,l})$, where $k,l\in \IN$, $k\leq l$ and $B^\eps_k := A^\eps_{\{a_k\}}$.
		Moreover, for $k\leq l$ and in view of Lemma \ref{le:150306-1820}, we can define
		\begin{equation*}
		g_{k,l} := f^{-1}_{\{a_l\}, \left[a_k,a_l\right]}\circ f_{\{a_k\}, \left[a_k,a_l\right]}.
		\end{equation*}
		
		Then, the inclusions $i_k:\; (N_1(\{a_k\}), N^{-\eps}_2(\{a_k\}) \to (N_1, N^{-\eps}_2)$ induce an isomorphism
		\begin{equation*}
		g:\; \dirlim (B^{\eps}_k, g_{k,l}) \to \Hom_*[N_1, N^{-\eps}_2],\quad \rep{k,x}\mapsto \Hom_*(p\circ i_k)(x),
		\end{equation*}
		where $p$ is the canonical projection onto the quotient space as given by Lemma \ref{le:141210-1901}.
	\end{lemma}
	
	\begin{proof}
		First of all, we need to prove that $g$ is well defined. Let there be given two
		representations
		$\rep{k,x} = \rep{l,x'}$ of the same element in $\dirlim (B^\eps_k, g_{k,l})$, that is, $g_{k,l}(x) = x'$.
		The following diagram with inclusion induced morphisms is commutative.
		\begin{equation*}
		\xymatrix{
			B^\eps_k \ar[r] \ar@/^2pc/[rr]^-{g_{k,l}} \ar[rd]_-{\Hom_*(i_k)}& 
			\Hom_*(N_1(\left[a_k,a_l\right]), N^{-\eps}_2(\left[a_k,a_l\right])) \ar[d] &
			B^\eps_l \ar[l] \ar[ld]^-{\Hom_*(i_l)}\\
			&\Hom_*(N_1, N_2)
		}
		\end{equation*}
		Consequently, $g$ is well defined.
		
		Let the isomorphism $j:\;\dirlim (A^\eps_\alpha, f_{\alpha,\beta}) \to \Hom_*[N_1, N^{-\eps}_2]$
		be given by Lemma \ref{le:140113-1642} with $(N_1, N_2)$
		replaced by $(N_1, N^{-\eps}_2)$. It is clear
		that $j(\rep{a_k,x}) = g(\rep{k,x})$ for all $\rep{k,x}\in\dirlim(B^{\eps}_k, g_{k,l})$.
		Letting $y\in \Hom_*(N_1, N^{-\eps}_2)$, there exists $\rep{\alpha, x}\in \dirlim(A^\eps_k,f_{\alpha,\beta})$
		such that $j(\rep{\alpha,x}) = y$. We can assume without loss of generality
		that $\alpha = \left[a_k,a_l\right]$ for $k,l\in\IN$ with $k\leq l$. 
		It follows from Lemma \ref{le:150306-1820} that $j(\rep{\alpha,x}) = j(\rep{\{a_l\}, x'}) = y$
		for some $x'\in A^\eps_{\{a_k\}} = B^\eps_k$. Thus, $g$ is an epimorphism.
		
		Assume that $g(\rep{k,x}) = 0$. Since $j$ is an isomorphism, it follows
		that $\rep{\{a_k\},x}=0$, so there exists a compact interval $\left[a,b\right]$
		with $a\leq a_k\leq b$ such that $f_{\{a_k\}, \left[a,b\right]}(x) = 0$. We even
		have $f_{\{a_k\}, \left[a,a_l\right]}(x) = f_{\left[a_k,a_l\right], \left[a,a_l\right]} \circ f_{\{a_k\},\left[a_k,a_l\right]}(x) = 0$ provided that $b\leq a_l$, so $f_{\{a_k\},\left[a_k,a_l\right]}=0$.
		From Lemma \ref{le:150306-1820}, one obtains $g_{k,l}(x) = f^{-1}_{\{a_l\},\left[a_k,a_l\right]}\circ f_{\{a_k\}, \left[a_k,a_l\right]}(x) =  0$, so
		$\rep{k,x} = \rep{l,g_{k,l}(x)} = \rep{l,0} = 0$. We have
		proved that $g$ is a monomorphism.
	\end{proof}
\end{section}


\begin{section}{Uniformly connected attractor-repeller decompositions}
	In analogy to the previous section, let $V\subset Y\times X$ and define\footnote{for arbitrary spaces $Y$ and $X$, in particular also $Y=\IR^+$}
	\label{sym:fibres}
	\begin{align*}
	V(y) &:= \{x:\; (y,x)\in V\}\\
	V(U) &:= V\cap (U\times X)&\text{where } U\subset Y.
	\end{align*}
	
	\begin{definition}
		\index{uniformly connected attractor-repeller decomposition}
		Let $(y_0,K,A,R)$ be an attractor-repeller decomposition. We say
		that $A$ and $R$ are {\em not uniformly connected} (in $K$) if there exists
		an $y\in\omega(y_0)$, an open neighbourhood $U_A$ of $A$ in $\Hull^+(y_0)\times X$
		and an open neighbourhood $U_R$ of $R$ in $\Hull^+(y_0)\times X$
		such that $U_A\cap U_R = \emptyset$ and $K(y)\subset U_A(y)\cup U_R(y)$.
		
		Otherwise, $(y_0,K,A,R)$ is called {\em uniformly connected}.
	\end{definition}
	
	The following theorem is the main result of this section, and the rest
	of the section is devoted to its proof. The strategy is to exploit 
	Lemma \ref{le:150114-1618} together with the assumption
	that $A$ and $R$ are not uniformly connected. 
	
	\begin{theorem}
		\label{th:141205-1538}
		Let $(y_0,K,A,R)$ be an attractor-repeller decomposition, and let
		there exist a strongly admissible isolating neighbourhood $N\subset\Hull^+(y_0)\times X$
		for $K$.
		
		The connecting homomorphism of the associated attractor-repeller sequence
		is trivial if $(y_0, K, A, R)$ is not uniformly connected.
	\end{theorem}
	
	\begin{lemma}
		\label{le:150109-1812}
		Let $(y_0, K, A, R)$ be an attractor-repeller decomposition such 
		that $A$ and $R$ are not uniformly connected. Suppose that $K$ is compact.\footnote{This follows from the
			assumptions of Theorem \ref{th:141205-1538}}
		
		Then there are $y'\in \omega(y_0)$, an open neighbourhood $U_A\subset \Hull^+(y_0)\times X$ of $A$,
		an open neighbourhood $U_R\subset \Hull^+(y_0)\times X$ of $R$ and a neighbourhood $U$ of $y'$ in $\Hull^+(y_0)$ such that $U_A\cap U_R = \emptyset$ and
		$K(y) \subset U_A(y)\cup U_R(y)$ for all $y\in U$.
	\end{lemma}
	
	\begin{proof}
		Since $A$ and $R$ are not uniformly connected, there must exist a $y'\in\omega(y_0)$, an
		open neighbourhood $U_A\subset\Hull^+(y_0)\times X$ of $A$ and an open neighbourhood
		$U_R\subset \Hull^+(y_0)\times X$ of $R$ such that
		$U_A\cap U_R=\emptyset$ and $K(y')\subset U_A(y')\cup U_R(y')$.
		
		Suppose that the lemma does not hold. Then
		there is a sequence $(y_n, x_n)\in K$ such that $y_n\to y'$ and
		$x_n\in K(y_n)\setminus (U_A\cup U_R)$.
		Due to the compactness of $K$, we may assume without loss of generality
		that $(y_n,x_n)\to (y',x_0)\in K$. Thus, $(y', x_0)\in K\setminus (U_A\cup U_R)$,
		which is a contradiction.
	\end{proof}
	
	Let $y'\in \omega(y_0)$, $U\subset \Hull^+(y_0)$ a closed neighbourhood of $y'$, $U_A\subset \Hull^+(y_0)\times X$ an open neighbourhood
	of $A$ and $U_R\subset \Hull^+(y_0)\times X$ an open neighbourhood of $R$
	for which the conclusions of Lemma \ref{le:150109-1812} hold. 
	There is a sequence $t_n\to\infty$ in $\IR^+$
	such that $a_n := y^{t_n}_0\in U$ for all $n\in\IN$. By the choice of
	$U$, one has $U_A(a_n)\cap U_R(a_n) = \emptyset$ and $K(a_n)\subset U_A(a_n)\cup U_R(a_n)$
	for all $n\in\IN$. 
	
	\begin{lemma}
		\label{le:150112-1631}
		Let $N'_\eps := \cl_{\Hull^+(y_0)\times X}\bigcup_{(y,x)\in K} B_\eps(y,x)$.
		
		There is a real $\eps_0>0$ such that for all $\eps<\eps_0$, 
		$N'_\eps \subset\Hull^+(y_0)\times X$ is an isolating neighbourhood for $K$ such
		that $N'_\eps(U) \subset U_A(U)\cup U_R(U)$.
	\end{lemma}
	
	\begin{proof}
		It is sufficient to prove that for all $\eps>0$ sufficiently small, $N'_\eps(U)$ is an isolating neighbourhood
		for $K$ in $\Hull^+(y_0)\times X$ and $N'_\eps(U)\subset U_A\cup U_R$.
		$K$ is (by assumption) an isolated invariant set admitting a strongly admissible isolating neighbourhood,
		so for small $\eps>0$, $N'_\eps$ is an isolating neighbourhood for $K$.
		
		Suppose that $N'_\eps(U)\subset U_A\cup U_R$ does not hold for small $\eps>0$.
		Using the compactness of $K$ and the closedness assumption on $U$, one concludes that
		there is a point $(y,x)\in K(U)\setminus (U_A\cup U_R) = \emptyset$, which is a contradiction.
	\end{proof}
	
	Fix an isolating neighbourhood $N'\subset\Hull^+(y_0)\times X$ for $K$ for which
	the conclusions of Lemma \ref{le:150112-1631} hold. 
	Recall that $r:\;\IR^+\times X\to \Hull^+(y_0)\times X$ is defined by $r(t,x) := (y^t_0,x)$.
	By using Lemma 4.3 in \cite{article_naci_1}, one obtains an index triple $(N_1, N_2, N_3)$
	for $(y_0, K, A, R)$ with $N_1\subset r^{-1}(N')$.

	\begin{lemma}
		\label{le:170918-1737}
		Letting $\hat N'_\eps:= r^{-1}(N'_\eps)$, there exist
		reals $\eps>0$, $T>0$ such that
		\begin{equation}
			\label{eq:170811-1822}
			\hat N'_\eps\setminus N^{-T}_2(N_1) \subset r^{-1}(U_R).
		\end{equation}
	\end{lemma}
	
	\begin{proof}
	Suppose to the contrary that there exists sequences $(\eps_n)_n$ in $\IR^+$ and
	$(t_n,x_n)$ in 
	$\cl_{\IR^+\times X} \left(\hat N'_{\eps_n}\setminus N^{-n}_2(N_1)\right)$ with $\eps_n\to 0$
	and $r(t_n,x_n)\not\in U_R$ for all $n\in\IN$.
	As $\eps_n\to 0$, there is a subsequence $(t_{n(k)}, x_{n(k)})_k$ such that $r(t_{n(k)},x_{n(k)})\to (y,x)\in K\setminus U_R$. 
	
	Let $N_R$ be an isolating neighbourhood for $R$ in $\Hull^+(y_0)\times X$ with $N_1\setminus N_2\subset r^{-1}(N_R)$. We have $(y,x)\pi t\in N_R$ for all $t\in\IR^+$, so $(y,x)\in R\setminus U_R = \emptyset$.
	\end{proof}
	
	\begin{lemma}
		\label{le:170918-1736}
		For arbitrary $\eps>0$,
		there is an index triple $(L_1, L_2, L_3)$ for $(y_0, K, A, R)$
		such that $(L_1, L_2, L_3)\subset (\hat N'_\eps\cap N_1, N_2, N^{-T}_3)$ for some $T=T(\eps)\in\IR^+$.
	\end{lemma}
	
	\begin{proof}
		By choosing $\eps$ smaller if required we can assume without loss of generality
		that $\hat N'_\eps$ and $N_3$ are disjoint. By virtue of Theorem 2.9 in \cite{article_naci_1}, 
		there is an index pair $(L_1, L_3)$ for $(y_0, K)$ with $(L_1, L_3)\subset (\hat N'_\eps\cap N_1, N^{-T}_3)$
		for real $T$ sufficiently large. Setting $L_2 := L_1\cap N_2$, it remains to prove
		that $(L_2, L_3)$ is an index pair for $(y_0, A)$.
		
		\begin{enumerate}
			\item[(IP2)] Let $x\in L_2$ and $x\chi t\not\in L_2$ for some $t>0$. We must have $x\chi t\not\in N_2$ or
			$x\chi t\not\in L_1$. Thus either $x\chi s\in L_3$ for some $s\in\left[0,t\right]$ or $x\chi\left[0,t\right]\subset L_1$ and $x\chi s\not\in N_1$ for some $s\in\left[0,t\right]$.
			The second case leads immediately to a contradiction since $L_1\subset N_1$.
			\item[(IP3)] Let $x\in L_2$ and $x\chi t\not\in L_2$ for some $t>0$. Either $x\chi t\not\in N_2$,
			so $x\chi s\in (\IR^+\times X)\setminus N_2$ for
			some $s\in\left[0,t\right]$, or $x\chi t\not\in L_1$. In the second case,
			it follows that $x\chi s\in (\IR^+\times X)\setminus L_2$ for some $s\in\left[0,t\right]$.
			\item[(IP4)] There is an isolating neighbourhood $M^1$ (resp. $M^2$) for $K$ (resp. $A$)
			with $\cl_{\IR^+\times X} (L_1\setminus L_3) \subset r^{-1}(M^1)$ (resp. $\cl_{\IR^+\times X} (N_2\setminus N_3) \subset r^{-1}(M^2)$). Another isolating neighbourhood
			for $A$ is  $M^3 := M^1\cap M^2$. One has $\cl_{\IR^+\times X} (L_2\setminus L_3) \subset \cl_{\IR^+\times X} (L_1\setminus L_3)
			\subset r^{-1}(M^1)$ and, using the fact that $L_2$ and $N_3$ are disjoint, $\cl_{\IR^+\times X} (L_2\setminus L_3) \subset \cl_{\IR^+\times X} (N_2\setminus N_3)
			\subset r^{-1}(M^2)$. Combining the inclusions yields
			$\cl_{\IR^+\times X} (L_2\setminus L_3) \subset r^{-1}(M^1) \cap r^{-1}(M^2) = r^{-1}(M^1\cap M^2)$.
			\item[(IP5)] There is a neighbourhood $U^1$ (resp. $U^2$) of $K$ (resp. $A$) such that
			$r^{-1}(U^1)\subset L_1\setminus L_3$ (resp. $r^{-1}(U^2)\subset N_2\setminus N_3\subset N_2$). 
			$U^3:=U^1\cap U^2$ is a neighbourhood of $A$ and $r^{-1}(U^3)\subset (L_1\setminus L_3)\cap N_2$.
		\end{enumerate}
	\end{proof}
	
	By redefining $(N_1, N_2, N_3)$ and using Lemma \ref{le:170918-1736} together with Lemma \ref{le:150112-1631},
	Lemma \ref{le:170918-1737}
	and Lemma 4.5 in \cite{article_naci_1}, one can assume that
	$(N_1, N_2, N_3)$ is an index triple for $(y_0, K, A, R)$ such that
	\begin{align}
		N_1(r^{-1}(U)) &\subset r^{-1}(U_A) \cap r^{-1}(U_R) \label{eq:170919-1647}\\
		N_1\setminus N_2 &\subset r^{-1}(U_R) \notag 
	\end{align}
	and thus also
	\begin{equation}
		\label{eq:170919-1442}
		N_1\cap r^{-1}(U_A) \subset N_2
	\end{equation}

	\begin{lemma}
		\label{le:150112-1725}
		Let $(N_1, N_2, N_3)$ be an index triple for $(y_0, K, A, R)$. Then
		there is a set $N'_3\supset N_3$ such that
		$(N_1, N'_3)$ is a regular index pair for $(y_0,K)$, $N'_3\subset N^{-\tau}_3(N_1)$ for some $\tau\geq 0$ and
		$(N^{-\tau}_2(N_1), N'_3)$ is a regular index pair for $(y_0,A)$.
	\end{lemma}
	
	\begin{proof}
		It follows from Lemma \ref{le:140107-1916} that there exist a set $N'_3\supset N_3$ and a constant $\tau\in\IR^+$
		such that $(N_1, N'_3)$ is a regular index pair for $(y_0, K)$ and $N_3\subset N'_3\subset N^{-\tau}_3(N_1)$. This means in particular that the 
		(inner) exit time $T:\; N_1\to \left[0,\infty\right]$, $T(x):=\sup\{t\in\IR^+:\; x\chi_{y_0}\left[0,t\right]\subset N_1\setminus N'_3\}$
		is continuous. 
		
		One needs to prove that $(N^{-\tau}_2(N_1), N'_3)$ is an index pair.
		\begin{enumerate}
			\item[(IP2)] Let $x\in N^{-\tau}_2(N_1)$ and $x\chi_{y_0} t\not\in N^{-\tau}_2(N_1)$
			for some $t\geq 0$. One cannot have $x\chi_{y_0} \left[0,t\right]\subset N_1$,
			so $x\chi_{y_0} s\in N'_3$
			for some $s\leq t$.
			\item[(IP3)] Let $x\in N'_3$ and $x\chi_{y_0} t\not\in N'_3$. It follows
			that $x\chi_{y_0} s\in (\IR^+\times X)\setminus N_1\subset (\IR^+\times X)\setminus N^{-\tau}_2(N_1)$ for some $s\in\left]0,t\right]$
			because $(N_1, N'_3)$ is an index pair.
		\end{enumerate}
		
		By Lemma 4.5 in \cite{article_naci_1}, $(N^{-\tau}_2(N_1), N_3)$ and $(N^{-\tau}_2(N_1), N^{-\tau}_3(N_1))$ are
		index pairs for $(y_0, A)$. By using the sandwich lemma \cite[Lemma 2.8]{article_naci_1}, one concludes that $(N^{-\tau}_2(N_1), N'_3)$ is an index pair for $(y_0, A)$.
		
		Finally, the exit time with respect to the index pair $(N^{-\tau}_2(N_1), N'_3)$ is the restriction
		of $T$ to $N^{-\tau}_2(N_1)$ and therefore continuous that is, the index pair is regular.
	\end{proof}
	
	Having proved Lemma \ref{le:150112-1725}, we can assume without loss of generality
	that $(N_1, N_2, N_3)$ is an index triple, satisfies \eqref{eq:170919-1647}, \eqref{eq:170919-1442}, 
	and $(N_1, N_3)$ as well as $(N_2, N_3)$ are regular index pairs\footnote{One could say that $(N_1, N_2, N_3)$
		is a regular index triple.}.
	
	By using \cite[Theorem 2.9]{article_naci_1}, we can further assume that there exists
	an index pair $(L_1, L_2)$ for $(y_0, A)$ such that $L_1\subset r^{-1}(U_A)$.

	We are now in a position to complete the proof of Theorem \ref{th:141205-1538}. 
	In view of the hypotheses of Lemma \ref{le:150114-1618}, one chooses a small real $\eps>0$,
	which is unrelated to the previous uses of this letter.
	
	Let $\iota:\;\Hom_*[L_1, L_2]\to \Hom_*[N^{-\eps}_2, N^{-\eps}_3]$ be inclusion induced.
	It is clear that $\iota$ is an isomorphism. Moreover, the following long exact sequence
	\begin{equation*}
	\xymatrix@1@C=6mm@R=15mm{
		\ar[r] &
 		\Hom_*[N_1,N^{-\eps}_3] \ar[r] &
		\Hom_*[N_1,N^{-\eps}_2] \ar[r]^-{\delta} \ar[rd]^-{\iota^{-1}\circ\delta}&
		\Hom_{*-1}[N^{-\eps}_2,N^{-\eps}_3] \ar[r]^-{i} &
		\Hom_{*-1}[N_1,N^{-\eps}_3] \ar[r] & \\
		&&&
		\Hom_{*-1}[L_1, L_2] \ar[u]^-\iota
		&&		
	}
	\end{equation*}
	associated with the index triple $(N_1, N^{-\eps}_2, N^{-\eps}_3)$ gives rise to the connecting homomorphism. 
	The above sequence is exact, so in order to prove $\delta=0$
	it is sufficient to prove that $i\circ \iota$ and thus also $i$ is a monomorphism, where $i:\; \Hom_*[N^{-\eps}_2,N^{-\eps}_3] \to \Hom_*[N_1,N^{-\eps}_3]$
	is inclusion induced.

	Let $(a_n)_n$ be a sequence of reals such that $a_n\to\infty$ and $y^{a_n}_0\in U$ for all $n\in\IN$.
	By using Lemma \ref{le:150114-1618}, a commutative diagram
	\begin{equation*}
	\xymatrix{
		\Hom_*[L_1, L_2] \ar[r]^-{i\circ \iota}  & \Hom_*[N_1, N^{-\eps}_3] \\
		\dirlim(B^\eps_k, g_{k,l}) \ar[r]^-{j} \ar[u]^g & \dirlim(\tilde B^\eps_k, \tilde g_{k,l}) \ar[u]^{\tilde g} 
	}
	\end{equation*}	
	is obtained. Recall that by definition $B^\eps_k := \Hom_*(L_1(\{a_k\}), L_3(\{a_k\})$
	and $\tilde B^\eps_k := \Hom_*(N_1(\{a_k\}), N^{-\eps}_3(\{a_k\}))$. Let $j_k:\; B^\eps_k \to \tilde B^\eps_k$
	be inclusion induced and set $j(\rep{k,x}) := \rep{k,j_kk(x)}$ for $x\in B^\eps_k$. We omit the proof that $j$ is well-defined.
	
	By Lemma \ref{le:150114-1618}, $g$ and $\tilde g$ are isomorphisms. Thus it is sufficient to prove
	that $j$ is a monomorphism. Suppose $j(\rep{k_0,x}) = \rep{k_0,j_{k_0}(x}) = \rep{k_0,y} = 0$ for $x\in B^\eps_{k_0}$, $y\in \tilde B^\eps_{k_0}$ and $k_0\in N$. There is an $l\in\IN$, $l\geq k_0$ such that $\tilde g_{k_0,l}(y) = 0$. 
	Furthermore, $j_l \circ g_{k_0,l}(x) = \tilde g_{k_0,l}(y) = 0$. We can hence assume without loss of generality
	that $j_{k_0}(x) = 0$.
	
	For brevity, a couple of notational shortcuts are introduced: $M_1 := N_1(\{a_{k_0}\})$, $M_2 := N^{-\eps}_2(\{a_{k_0}\})$,
	$M_3 := N^{-\eps}_3(\{a_{k_0}\})$, $\hat U_A := r^{-1}(U_A)$ and $\hat U_R := r^{-1}(U_R)$.
	It follows from \eqref{eq:170919-1647} and the choice of the sequence $a_k$ that
	$\Hom_*(M_1\cap \hat U_A, M_3\cap \hat U_A) \oplus \Hom_*(M_1\cap \hat U_R, M_3\cap \hat U_R)
	\iso \Hom_*(M_1, M_3)$. Let the projection $p:\; \Hom_*(M_1, M_3)\to \Hom_*(M_1\cap\hat U_A, M_3\cap\hat U_A)$
	be defined by the above direct sum decomposition.
	
	We obtain once again a commutative diagram with inclusion induced homomorphisms and the projection $p$ 
	introduced above.
	\begin{equation*}
	  \xymatrix{
	  	\Hom_*(L_1(\{a_{k_0}\}), L_2(\{a_{k_0}\})) \ar[r]^-{h_1} \ar@/^8mm/[rr]^-{j_{k_0}}&
	  	\Hom_*(M_2, M_3) \ar[r]^-{h_2} & 
	  	\Hom_*(M_1, M_3) \ar[d]^-p \\
	  	&&
	  	\Hom_*(M_1\cap\hat U_A, M_3\cap\hat U_A) \ar[lu]^-{h_3}
	  }
	\end{equation*}
	It follows that $h_1(x) = h_3\circ p \circ h_2\circ h_1(x) = 0$. Moreover,
	$h_1(x) = \iota(g([k,x]))$, implying
	that $[k,x]=0$ in $\dirlim(B^\eps_k, g_{k,l})$ since $\iota$ and $g$ are isomorphisms.
	The proof of Theorem \ref{th:141205-1538} is complete.
\end{section}

\providecommand{\bysame}{\leavevmode\hbox to3em{\hrulefill}\thinspace}
\providecommand{\MR}{\relax\ifhmode\unskip\space\fi MR }
\providecommand{\MRhref}[2]{%
	\href{http://www.ams.org/mathscinet-getitem?mr=#1}{#2}
}
\providecommand{\href}[2]{#2}

\end{document}